\documentclass[11pt]{article}
 \usepackage{amsmath,amssymb,amsthm}
 \RequirePackage[dvips]{graphicx}
  \usepackage{cite}

 \def\draw #1 by #2 (#3){
  \vbox to #2{
    \hrule width #1 height 0pt depth 0pt
    \vfill
    \special{picture #3} 
    }
  }

 \def\scaleddraw #1 by #2 (#3 scaled #4){{
  \dimen0=#1 \dimen1=#2
  \divide\dimen0 by 1000 \multiply\dimen0 by #4
  \divide\dimen1 by 1000 \multiply\dimen1 by #4
  \draw \dimen0 by \dimen1 (#3 scaled #4)}
  }

\newtheorem{theorem}{Theorem}[section]
\newtheorem{example}{Example}
\newtheorem{problem}[example]{Problem}
\newtheorem{defin}[theorem]{Definition}
\newtheorem{lemma}[theorem]{Lemma}
\newtheorem{corollary}[theorem]{Corollary}

\newtheorem{remark}[theorem]{Remark}
\newtheorem{nt}{Note}

\allowdisplaybreaks[4]  

 \newcommand{\singlespacing}{\let\CS=\@currsize\renewcommand{\baselinestretch}{1}\tiny\CS}
 \newcommand{\oneandahalfspacing}{\let\CS=\@currsize\renewcommand{\baselinestretch}{1.25}\tiny\CS}
 \newcommand{\doublespacing}{\let\CS=\@currsize\renewcommand{\baselinestretch}{1.35}\tiny\CS}

 \newtheorem{rule-def}[theorem]{Rule}

\begin{document}
\baselineskip 16pt
 \newcommand{\la}{\lambda}
 \newcommand{\si}{\sigma}
 \newcommand{\ol}{1-\lambda}
 \newcommand{\be}{\begin{equation}}
 \newcommand{\ee}{\end{equation}}
 \newcommand{\bea}{\begin{eqnarray}}
 \newcommand{\eea}{\end{eqnarray}}

  \baselineskip=0.23in

 \begin{center}
 {\Large \bf Hoffman-type Results for the Sum of $k$ Largest\\[2mm] Eigenvalues of a Graph}\\

  \vspace{10mm}

 {\large \bf Shaowei Sun$^{a}$, Mengyao Guo$^{a}$, Hongyan Ge$^{a}$, Kinkar Chandra Das$^{b,}\footnote{Corresponding author}$}

 \vspace{9mm}

 \baselineskip=0.20in
$^a${\it School of Science, Zhejiang University of Science and Technology,\\
 Hangzhou, Zhejiang, 310023, PR China \/}\\
{\rm E-mail:} {\tt sunshaowei2009@126.com,\,793126388@qq.com,\,398643629@qq.com}\\[2mm]
$^b${\it Department of Mathematics, Sungkyunkwan University,\\
Suwon 16419, Republic of Korea\/}\\[2mm]
{\rm E-mail:} {\tt kinkardas2003@gmail.com}

 \vspace{4mm}

 (Received September 21, 2026)

 \vspace{5mm}

 \end{center}

 \baselineskip=0.23in

 \begin{abstract}
Let $S_k(G)$ denote the sum of the $k$ largest eigenvalues
of a graph $G$. Motivated by the classical Hoffman program for the
spectral radius of a graph, we investigate an additive Hoffman-type
problem for $S_k(G)$. For each fixed $k\geq 2$ and sufficiently large
order $n$, we characterize all connected graphs satisfying
$S_k(G)<2k$. As a consequence, we prove that the path $P_n$ is the
unique minimizer of $S_k(G)$ among all connected graphs of order $n$.

\vspace*{2mm}

We further investigate the first Hoffman-type range
\[
2k\leq S_k(G)<2k+\sqrt{2+\sqrt5}-2.
\]
We completely characterize the non-tree graphs in this range and
reduce the tree case to several explicit families. The proofs combine Ky Fan's variational principle,
spectral estimates from vertex-disjoint subgraphs, structural results
for graphs with small spectral radius, and long-path arguments for
bounded-degree graphs.

 \vspace{5mm}


 \end{abstract}

 \vspace*{3mm}

 \baselineskip=0.30in

 \section{Introduction}
The eigenvalues of a graph contain important information about its
structure, and extremal problems concerning graph eigenvalues have
been extensively studied in spectral graph theory. Let
$\lambda_1(G)\geq \lambda_2(G)\geq \cdots \geq \lambda_n(G)$
be the adjacency eigenvalues of a graph $G$ of order $n$. For
$1\leq k\leq n$, define
\[
S_k(G)=\sum_{i=1}^{k}\lambda_i(G).
\]
In particular, $S_1(G)=\lambda_1(G)$ is the spectral radius of $G$,
one of the most fundamental spectral parameters. Thus, $S_k(G)$ may
be viewed as a natural extension of the spectral radius. By Ky Fan's
variational principle \cite{FAN2}, $S_k(G)$ can be expressed as the
maximum of the sum of the corresponding Rayleigh quotients over a set
of orthonormal vectors. It therefore extends the spectral radius while retaining a useful variational description.

Most previous work on sums of the largest eigenvalues of a graph and symmetric matrices has concerned upper bounds and maximization. Mohar \cite{MOHAR}
obtained a general upper bound for the sum of the $k$ largest
eigenvalues of a graph. This was further developed by Nikiforov
\cite{NIKIFOROV}, and several subsequent improvements and extensions were
obtained; see, for example, \cite{DMS,SMD2}. 
The case $k=2$ has received particular attention. Gernert asked whether $S_2(G)\leq n$ for every graph $G$ of order $n$. Nikiforov
\cite{NI} subsequently disproved this assertion and established linear bounds on the maximum of $S_2(G)$. Ebrahimi et al. \cite{EMNA} improved the upper bound and conjectured that $S_2(G)\leq \frac{8n}{7}$.
Recently, Kumar et al. \cite{KLMPT} proved
this conjecture for all graphs. Very recently, Huang and Wei
\cite{HW} completely determined the exact maximum of $S_2(G)$ for
every $n\geq 5$, together with all equality cases.

In contrast, considerably less is known about the corresponding
minimum problem. Kumar et al. \cite{KMPZ} studied the extremal values of
$S_2(G)$ for trees and more general convex combinations of the first
two adjacency eigenvalues. In \cite{SMD1}, we characterized all
connected graphs satisfying $S_2(G)<4$ and proved, as a consequence,
that the path graph is the unique minimizer of $S_2(G)$ among all
connected graphs of order $n\geq 467$.  These results naturally lead
to the following general question:
for a fixed $k$, which connected graph minimizes $S_k(G)$ when its order is sufficiently lagre?

A second motivation comes from the classical Hoffman
program for the spectral radius of a graph. Smith \cite{SMITH}
characterized the connected graphs with spectral radius at most $2$.
Hoffman \cite{Hoffman} determined all limit points below
$$\eta:=\sqrt{2+\sqrt5}.$$
Cvetkovi\'c, Doob and Gutman \cite{CDG}, followed by Brouwer and Neumaier
\cite{BN}, characterized the connected graphs with spectral radius between
$2$ and $\eta$.  For recent developments on the Hoffman program, we refer to \cite{WWBBW}.
These results reveal the fundamental role of the two thresholds
$2$ and $\eta$ in the characterization of graphs with small
spectral radius.

Since $S_1(G)$ is exactly the spectral radius of $G$, it is natural
to ask whether the Hoffman-type threshold phenomena persist for $S_k(G)$. A key observation in the present paper is that
the classical spectral-radius thresholds admit natural additive
counterparts for the sum of $k$ largest eigenvalues. In this
sense, our results may be viewed as an additive analogue of the
Hoffman program for eigenvalue sums. Therefore, we study the Hoffman program on the sum of $k$ largest eigenvalues of a graph. 

Our first main result gives a complete characterization of connected
graphs satisfying $S_k(G)<2k$ when the order is sufficiently large
relative to $k$. This characterization immediately leads to
the minimum problem: we prove that the path graph is the unique
minimizer of $S_k(G)$ among all connected graphs of sufficiently
large order. 
We then go one step further and consider the first range above the
threshold $2k$. We completely determine the non-tree graphs satisfying
$ 2k\leq S_k(G)<2k-2+\eta$ for sufficiently large order $n$. For trees, we reduce the problem to several explicit
families. Thus, the two spectral ranges
studied in this paper can be viewed as additive counterparts of the
classical ranges determined by the thresholds $2$ and
$\eta$ for the spectral radius.

The rest of the paper is organized as follows. In Section~2, we
introduce the notation and recall several preliminary results. In
Section~3, we develop the main tools for estimating $S_k(G)$ from
vertex-disjoint subgraphs, characterize the connected graphs with
$S_k(G)<2k$, and determine the unique minimizer. In Section~4, we
study graphs whose eigenvalues sums lie in the first range above $2k$.
Finally, some concluding remarks and an open problem are given in
Section~5.

 \section{Preliminaries}
 In this section, we introduce some notation and graph families used
throughout the paper, and recall several known results that will be
needed in the subsequent sections.

\subsection{Notation and graph families}
Throughout this paper, all graphs are finite, simple and undirected.
Let $G$ be a graph with vertex set
$V(G)=\{v_1,v_2,\ldots,v_n\}$ and edge set $E(G)$. The order and
size of $G$ are denoted by $n=|V(G)|$ and $m=|E(G)|$, respectively.
For a vertex $v\in V(G)$, let $N_G(v)$ and $N_G[v]$ denote its open
and closed neighborhoods, respectively, and let
$d_G(v)=|N_G(v)|$ be its degree. We write $\Delta(G)$, or simply
$\Delta$, for the maximum degree of $G$. 
For two vertices $u,v\in V(G)$, let
$d_G(u,v)$ denote the distance between $u$ and $v$. For $S\subseteq V(G)$,
let $G[S]$ denote the subgraph induced by $S$. When there is no
ambiguity, the subscript $G$ will be omitted.

We use $P_n$ and $C_n$ to denote the path and the cycle of order $n$,
respectively. Let $P_n=v_1v_2\cdots v_n$. The graph $Y_n$ is obtained
from $P_n$ by deleting $v_1v_2$ and adding $v_1v_3$. For $n\geq6$,
the graph $W_n$ is obtained from $P_n$ by deleting $v_1v_2$ and
$v_{n-1}v_n$, and adding $v_1v_3$ and $v_{n-2}v_n$.
For positive integers $a,\, b,\,c$, let $T_{a,b,c}$ denote the
T-shape tree obtained by identifying one endpoint from each of
$P_{a+1}$, $P_{b+1}$ and $P_{c+1}$. Let $U_t$ denote the graph
obtained from $C_t$ by attaching a pendant edge to one vertex of
$C_t$.

For $1\le m_1<\cdots<m_t\le p-2$ and positive integers $n_1,\ldots,n_t$,
let
\[
 P_{n_1,\ldots,n_t,p}^{m_1,\ldots,m_t}
\]
be the tree obtained from the path $P_p:0\sim1\sim\cdots\sim p-1$ by
attaching a path of $n_i$ vertices at $m_i$, for each $i$.
Its order is $p+\sum_i n_i$. These trees are open quipus. Whenever this
underlying path is chosen to be diametral, its parameters satisfy
\[
 n_i\le\min\{m_i,p-1-m_i\}\qquad(1\le i\le t).
\]
Similarly, let
\[
C_{n_1,n_2,\ldots,n_t,p}^{m_1,m_2,\ldots,m_t}
\]
denote the graph obtained from the cycle
$C_p:0\sim1\sim\cdots\sim p-1\sim 0 $
by attaching a path of $n_i$ vertices at the vertex $m_i$,
for each $i=1,2,\ldots,t$. Graphs of this form are closed quipus.
In particular, $U_t\cong C_{1,t}^{0}$.

\subsection{Preliminary results}

We now recall several known results that will be used throughout the
paper. We begin with Ky Fan's variational principle.

  \begin{lemma} {\rm \cite{FAN2}} \label{lm1}
 Let $M$ be a symmetric matrix with eigenvalues $\lambda_1\geq \lambda_2\geq \cdots\geq \lambda_n$. Then
 $\lambda_1+ \lambda_2+ \cdots+ \lambda_r= \sup\{u_1^TMu_1+u_2^TMu_2+ \cdots + u_r^TMu_r\}$ $(r = 1, 2,\ldots, n)$, where the
 supremum is taken over all orthonormal vectors $u_1, u_2,\ldots, u_r$.
 \end{lemma}

We next recall the Cauchy interlacing theorem.
\begin{lemma} {\rm \cite{SCH}} \label{lm6}
 Let $B$ be an $n\times n$ symmetric matrix and let $B_r$ be its $r\times r$ principal submatrix. Then, for $i=1,\,2,\ldots,\,r$,
 $$\lambda_{i}(B)\geq \lambda_i(B_r)\geq \lambda_{n-r+i}(B),$$
 where $\lambda_{i}(B)$ and $\lambda_i(B_r)$ denote the $i$-th largest eigenvalue of $B$ and $B_r$, respectively.
 \end{lemma}
 
As an immediate consequence, we obtain the following monotonicity
property of $S_k(G)$ with respect to induced subgraphs.
 \begin{corollary} \label{lm7} 
 Let $G$ be a graph, and let $H$ be its induced subgraph. Then $S_k(G)\geq S_k(H)$.
 \end{corollary}

The following result gives the relation between the spectral radii of
a graph and its subgraphs.
 \begin{lemma}  {\rm \cite{BOOK}} \label{lm3}
 If $H$ is a subgraph of a graph $G$, then $\lambda_1(H)\leq \lambda_1(G)$.
 \end{lemma}
 
An \emph{internal path} of a graph is a sequence of adjacent vertices
whose internal vertices have degree $2$ and whose two end-vertices
have degree greater than $2$; the two end-vertices are allowed to
coincide. Subdividing an edge of an internal path means inserting a
new vertex of degree $2$ into that edge.
The following is the Hoffman--Smith subdivision theorem.
\begin{lemma} {$($Hoffman and Smith’s subdivision theorem$)$\rm \cite{HS}} \label{lm8}
 Let $G$ be a graph with an internal path, and let $G^{\prime}$ be the graph obtained from $G$ by subdividing an edge on that path. If $G\ncong W_n$, then $\lambda_1(G^{\prime})<\lambda_1(G)$.
 \end{lemma}

We next recall the spectra of paths and cycles, together with a
standard estimate for the cosine function.
 \begin{lemma}  {\rm \cite{BOOK}} \label{lm2}
 The eigenvalues of path graph $P_n$ are
 $2\cos\Big(\frac{i\,\pi}{n+1}\Big)$ with corresponding eigenvector
 $$\left(\sin\Big(\frac{i\,\pi}{n+1}\Big),\,\sin\Big(\frac{2i\,\pi}{n+1}\Big),\ldots,\sin\Big(\frac{ni\,\pi}{n+1}\Big)\right)^T,~~i=1,\,2,\,3,\ldots,\,n.$$
 \end{lemma}
 
\begin{lemma}  {\rm \cite{BOOK}} \label{lm22}
 The eigenvalues of cycle graph $C_n$ are
 $2\cos\Big(\frac{2i\pi}{n}\Big)$, $i=1,\,2,\ldots,n$.
\end{lemma}
 
 \begin{lemma}  {\rm \cite{WS}}  \label{lm5}
For  all \(x \in (0,1)\), 
$$1 - \frac{1}{2}x^2 \le \cos x \le 1 - \frac{1}{2}x^2 + \frac{1}{24}x^4.$$
  \end{lemma}

We finally recall several structural results concerning graphs with
small spectral radius.

 \begin{lemma}  {\rm \cite{SMITH}}  \label{lm4}
 Let $G$ be a connected graph with order $n\geq 10$. Then\\
 (i) $\lambda_1(G)<2$ if and only if $G\in \{P_n,\,Y_n\}$;\\
  (ii) $\lambda_1(G)=2$ if and only if $G\in \{C_n,\,W_n\}$.
  \end{lemma}

 \begin{lemma} {\rm \cite{BN,CDG}} \label{lm9}
 Let $G$ be a connected graph with order $n\geq 18$. Then $\lambda_1(G)\in (2,\,\eta]$ if and only if 
 $$G\in \left\{T_{1,\,a,\,n-a-2},\,T_{2,\,2,\,n-5},\,P_{1,\,1,\,n-2}^{b,\,b+c}\right\},$$
 where $a\in [2,\,n-4]$, $b,\,c\in [1,\,n-4]$ and $(b,\,c)\neq (1,\,n-4)$ with the following condition 
 \[
2c\geq 
\begin{cases}
n-4 & \mbox{ for } b=1, \\
n-2 & \mbox{ for } b=2,\\
n-1 & \mbox{ for } b>2.
\end{cases}
\]
 \end{lemma}

 \section{Characterization of graphs with eigenvalue sum below $2k$}
 
 Throughout this section, $k\geq 2$ unless otherwise stated.
 The classical threshold $2$ for the adjacency spectral radius suggests
the additive threshold $2k$ for the sum of the $k$ largest eigenvalues. 
 In this section, we study connected graphs whose sum of the $k$
largest eigenvalues lies below the threshold $2k$.
We begin with the following observation.
 \begin{lemma} \label{lms1}Let $V_1,\ldots,V_k$ be nonempty, pairwise disjoint subsets of $V(G)$. Then 
 $$S_k(G) \geq \sum_{i=1}^k \lambda_1\left(G[V_i]\right).$$
 \end{lemma}

 \begin{proof} Let ${\bf x_i}$ be the unit eigenvector of $G[V_i]$ corresponding to its spectral radius $\lambda_1(G[V_i])$, for each $i = 1, 2, \dots, k$. Then we have ${\bf x_i}^T A(G[V_i]) {\bf x_i} = \lambda_1(G[V_i])$. Extend each ${\bf x_i}$ to a vector of dimensions $n$ (say, ${\bf y_i}$) by setting all entries outside $V_i$ to zero. Then 
 $${\bf y_i}^T A(G){\bf y_i}={\bf x_i}^T A(G[V_i]) {\bf x_i}=\lambda_1(G[V_i]).$$
 Since the sets $V_i$ are pairwise disjoint, the vectors ${\bf y_1}, {\bf y_2}, \dots, {\bf y_k}$ are unit and pairwise orthogonal. 
 Combining the above results with Lemma \ref{lm1}, we get 
 $$S_k(G) \ge \sum_{i=1}^k \frac{{\bf y_i}^T A(G) {\bf y_i}}{{\bf y_i}^T {\bf y_i}} = \sum_{i=1}^k \lambda_1(G[V_i]).$$
 This completes the proof.
\end{proof}

The preceding lemma immediately yields the following useful form for
vertex-disjoint subgraphs. Corollary~\ref{lms2} will be used repeatedly throughout this section. It allows us to obtain lower bounds for $S_k(G)$ by finding suitable vertex-disjoint subgraphs of $G$.
 \begin{corollary} \label{lms2} Let $G$ be a graph  and let  $G_1,G_2,\ldots,G_k$ be vertex-disjoint subgraphs of $G$. Then $$S_k(G) \geq \sum_{i=1}^k \lambda_1(G_i).$$
\end{corollary}

 \begin{proof} Since $G_1,G_2,\ldots,G_k$ be vertex-disjoint subgraphs of $G$, then $V(G_1),V(G_2),\ldots,V(G_k)$ are pairwise disjoint.
By taking $V_i=V(G_i)$ in Lemma \ref{lms1}, we get 
$$S_k(G) \geq \sum_{i=1}^k \lambda_1(G[V(G_i)]).$$
Moreover, each $G_i$ is a subgraph of $G[V(G_i)]$. Then by 
Lemma \ref{lm3}, it follows that $\lambda_1(G[V(G_i)])\geq \lambda_1(G_i)$. Combining the above results gives the desired result.
\end{proof}

We first record two elementary consequences that will be used to find
such vertex-disjoint subgraphs. The first controls graphs of large
maximum degree, while the second guarantees the existence of a long
path when the maximum degree is bounded.
 \begin{lemma} \label{lms3} Let $G$ be a connected graph of order $n$ with the maximum degree $\Delta$. If $n\geq (2k+1)^2+k$ and $\Delta \geq (2k+1)^2$, then $S_k(G)\geq 2k+1$.
 \end{lemma}

 \begin{proof}Since $\Delta \geq (2k+1)^2$ and $n \geq (2k+1)^2+k$, $K_{1,(2k+1)^2}$ and $(k-1)K_1$ are vertex-disjoint subgraphs of $G$. By Corollary \ref{lms2}, we obtain $$S_k(G) \geq \lambda_1(K_{1,(2k+1)^2} ) + (k-1)\lambda_1(K_1) =2k+1.$$
\end{proof}

 \begin{lemma} \label{lms4} Let $G$ be a connected graph with order $n\geq 3$ and the maximum degree $\Delta$. Then $G$ contains a path  $P_t$ with $t>2\log_\Delta n-2$.
\end{lemma}

 \begin{proof} Since $G$ is connected, it contains a spanning tree, say $T$. It is clear that $d_T(v_i)\leq d_G(v_i)\leq \Delta$. Let $P_t$ be the longest path of $T$. Choose a vertex $v_1\in V(P_t)$ such that $v_1$ is the center of $P_t$, we define $L_i = \{u \mid u \in V(T) \text{ and } d_T(u, v_1) = i\}$ for $i\geq 1$ and $h= \max_{v_i \in V(T)} \{d_T(v_1, v_i)\}$. Then $h\leq \frac{t}{2}$ and $|L_i| \leq \Delta^i$ for $1 \leq i \leq h$. It follows that
 $$n = 1+\sum_{i=1}^h |L_i| \le \sum_{i=0}^h \Delta^i = \frac{\Delta^{h+1} - 1}{\Delta - 1} < \Delta^{h+1},$$
 which implies that $h>\log_\Delta n-1$. Thus $t\geq 2h>2\log_\Delta n-2$. This completes the proof.
\end{proof}
For $k\ge2$ and $0<\varepsilon\le1$, define
\[
 r_\varepsilon=
 \left\lceil\frac{\pi}{\arccos(1-\varepsilon/(k-1))}\right\rceil-1.
\]
One can easily see that $r_\varepsilon\geq 1$.
 \begin{lemma} \label{lms5.1} Let $G$ be a connected graph of order $n$ that contains a $q$-vertex subgraph $H$ with $q\geq 4$. If $\lambda_1(H) \geq a + 2\varepsilon$ and $n \ge (2k+1)^{q-2+(q+1)(k-1)r_\varepsilon}$ with $a\in [2,\,3]$ and
 $\varepsilon \in (0,1]$, then $S_k(G) \geq 2k - 2 + a.$
\end{lemma}

 \begin{proof} Since $r_\varepsilon\geq 1$, from the given condition, we obtain $n\geq (2k+1)^{2q-1}>(2k+1)^2+k$ as $q\geq 4$. If $\Delta\geq (2k+1)^2$, then by Lemma \ref{lms3}, we have $S_k(G) \geq 2k +1\geq 2k- 2 + a$ as $a\leq 3$.
 Otherwise, $\Delta<(2k+1)^2$. By Lemma \ref{lms4} with the condition on $n$, $G$ contains a path $P_t$ such that 
 \begin{equation}
 t >2\log_\Delta n-2>2\log_{(2k+1)^2} (2k+1)^{q-2+(q+1)(k-1)r_\varepsilon}-2 = (q+1)(k-1)r_\varepsilon+q-4.\label{sgd1}
  \end{equation}
 Let $P_t[V(P_t) \setminus V(H)] \cong P_{m_1} \cup P_{m_2} \cup \dots \cup P_{m_j}$ such that $m_1\geq m_2\geq \cdots\geq m_j$. Since $|V(H)|= q$, it follows that $j \le q+1$. Combining this fact with the result in (\ref{sgd1}), we have
 $$m_1 \geq \frac{t - q}{j}\geq \frac{t - q}{q+1}>\frac{(q+1)(k-1)r_\varepsilon - 4}{q+1}=(k-1)r_\varepsilon-\frac{4}{q+1}.$$
 It follows that $m_1\geq (k-1)r_\varepsilon$ as $m_1$ is integer and $q\geq 4$. Thus $(k-1)P_{r_\varepsilon}$ and $H$ are vertex-disjoint subgraphs of $G$. By Corollary \ref{lms2} and Lemma \ref{lm2}, we have 
  $$S_k(G) \geq \lambda_1(H) + (k-1)\lambda_1(P_{r_\varepsilon}) \geq a + 2\varepsilon + 2(k-1)\cos\left(\frac{\pi}{r_\varepsilon+1}\right).$$
  From the definition of $r_\varepsilon$, we get $\cos\left(\frac{\pi}{r_\varepsilon+1}\right)\geq 1-\frac{\varepsilon}{k-1}$, which implies that $S_k(G)\geq 2k-2+a$.
 This completes the proof.
\end{proof}
We now define $\theta_\varepsilon(k)$ as follows:
$$\theta_\varepsilon(k)=
\begin{cases}
0 & \text{if } k=2,\\[2mm]
\displaystyle
\left\lceil
\frac{\pi}{\arccos\left(1-\frac{\varepsilon}{k-2}\right)}
\right\rceil-1
& \text{if } k\ge 3.
\end{cases}$$
We next extend the preceding result to two vertex-disjoint subgraphs.
 \begin{lemma} \label{lms5.2} Let $G$ be a connected graph of order $n$ that contains two vertex-disjoint subgraphs $H_1$ and $H_2$, where $H_1$ has $q_1$ vertices and $H_2$ has $q_2$ vertices with $q_1+q_2\geq 5$. If $\lambda_1(H_1)+\lambda_1(H_2) \geq b + 2\varepsilon$ and $n \ge (2k+1)^{q_1+q_2-3+(q_1+q_2+1)(k-2)\theta_\varepsilon}$ with $b\in [4,\,5]$ and
 $\varepsilon \in (0,1]$, then $S_k(G) \geq 2k - 4 + b.$
\end{lemma}

\begin{proof}
For $k=2$, then $S_2(G)\ge\lambda_1(H_1)+\lambda_1(H_2)$ by Corollary \ref{lms2}, and the conclusion follows. For $k\ge3$, we have $\theta_\varepsilon\ge1$. The large-degree case follows
from Lemma~\ref{lms3}, since $b\le 5$. Otherwise,
Lemma~\ref{lms4} gives a path of order
\[
 t>(q_1+q_2+1)(k-2)\theta_\varepsilon+q_1+q_2-5.
\]
Remove the vertices of $H_1\cup H_2$ from this path. One remaining component
has order greater than $(k-2)\theta_\varepsilon-5/(q_1+q_2+1)$, and therefore at least
$(k-2)\theta_\varepsilon$. Select $k-2$ disjoint copies of $P_{\theta_\varepsilon}$ in it.
Corollary~\ref{lms2} now gives
\[
 S_k(G)\ge\lambda_1(H_1)+\lambda_1(H_2)
 +2(k-2)\cos\frac{\pi}{\theta_\varepsilon+1}\ge2k-4+b.
\]
\end{proof}

We are now ready to characterize the connected graphs satisfying
$S_k(G)<2k$. We first consider graphs containing a cycle.

 \begin{lemma} \label{lms6} Let $G$ be a connected graph with $m \geq n \geq (2k+1)^{10+13(k-1)r_{0.021}}$.
Then $S_k(G)<2k$ if and only if $G\cong C_n$.
\end{lemma}

 \begin{proof}
 \noindent\textbf{Sufficiency:} If $G\cong C_n$, then by Lemma \ref{lm22}, we have $\lambda_1(G)=2$ and $\lambda_i(G)<2$ for $i\geq 2$. Thus $S_k(G)<2k$.
 \medskip 

 \noindent\textbf{Necessity:} Suppose that $S_k(G)<2k$. We have to prove that $G\cong C_n$. Let $C_t$ be a shortest cycle of $G$. If $t=n$, then $G\cong C_n$, we are done. Otherwise, $t\leq n-1$. Then $U_t$ is a subgraph of $G$. First we assume that $t \geq 11$. Then $T_{1,5,5}$ is a subgraph of $U_t$. Note that $\lambda_1(T_{1,5,5})\geq 2.042$ by SAGE \cite{SAGE}.
 Setting $H=T_{1,5,5}$, $a=2$ and $\varepsilon=0.021$ in Lemma \ref{lms5.1} with the given condition on $n$, we have $S_k(G) \geq 2k$, a contradiction as $S_k(G)<2k$.

 Next we assume that $3\leq t \leq 10$. Note that  $\lambda_1(U_{t})\geq \lambda_1(U_{10}) \geq 2.058$ by SAGE \cite{SAGE}.
 Setting $H=U_t$, $a=2$ and $\varepsilon=0.021$ in Lemma \ref{lms5.1}, together with the fact that
 $$n\geq (2k+1)^{10+13(k-1)r_{0.021}}>(2k+1)^{t-1+(t+2)(k-1)r_{0.021}},$$
 we get $S_k(G) \geq 2k$, again a contradiction as $S_k(G)<2k$.
 This completes the proof.
 \end{proof}

We next consider the tree case.
 \begin{lemma} \label{lms7}
 Let $T$ be a tree with $n \geq (2k+1)^{20+23(k-1)r_{0.0097}}$. Then $S_k(T)<2k$ if and only if $T \in \{Y_n, W_n, P_n\}$.
 \end{lemma}
 
 \begin{proof}
 \noindent\textbf{Sufficiency:} If $T \in \{Y_n, W_n, P_n\}$, then $\lambda_1(T) \leq 2$ by Lemma \ref{lm4}. It is well known that $\lambda_1(G)$ has multiplicity $1$ for any connected graph $G$. Thus $\lambda_2(T)<\lambda_1(T)\leq 2$, which implies that $S_k(T)<2k$.

 \vspace{0.5em}
  \noindent\textbf{Necessity:} Suppose that $S_k(T) < 2k$. We have to prove that $T \in \{Y_n, W_n, P_n\}$. Let $v_0$ be a vertex such that  $d(v_0) = \Delta$ and let $L_i = \{u \mid u \in V(T) \text{ and } d_T(u, v_0) = i\}$ for $i\geq 1$. First we assume that $\Delta \ge 4$. If $\Delta=n-1$, then $T$ is isomorphic to a star graph, and hence 
  $$S_k(T)=\sqrt{n-1}>2k$$
  as $n \geq (2k+1)^{20+23(k-1)r_{0.0097}}$.

  \vspace*{2mm}
  
  Otherwise, $\Delta\leq n-2$. Then $|L_1| \geq 4$ and $|L_2| \geq 1$, which implies that  $T$ contains a subgraph $T'$ which is obtained from $K_{1,\,4}$ by adding a pendant edge to one of its pendant vertices. By SAGE \cite{SAGE}, $\lambda_1(T') \geq 2.074$. 
   Setting $H=T'$, $a=2$ and $\varepsilon = 0.037$ in Lemma \ref{lms5.1}, together with the fact that
   $$n \geq (2k+1)^{20+23(k-1)r_{0.0097}}> (2k+1)^{4+7(k-1)r_{0.037}},$$
   we get $S_k(T)\geq 2k$, a contradiction as $S_k(T) < 2k$. 

   \vspace*{2mm}
   
   Next we assume that $\Delta \leq 3$. For $\Delta = 2$, then $T \cong P_n$, and the conclusion holds trivially. It remains to consider the case  $\Delta = 3$. Let $L_1=\{v_1,\,v_2,\,v_3\}$, which is the neighbor set of $v_0$. Let $P_t$ be a longest path of $T$. 
   As $n \geq (2k+1)^{8+11(k-1)r_{0.0097}}$, by Lemma \ref{lms4}, we obtain 
    $$t>2\log_3n-2> 20+23(k-1)r_{0.0097}>36.$$
    This implies that  there exists a vertex $u$ such that  $d(u, v_{0})= 18$.  Without loss of generality, we assume $d(u, v_{1})= 17$. If $d(v_{2})d(v_{3})\geq 2$, then $T_{1,2,18}$ is a subgraph of $T$. Note that $\lambda_1(T_{1,2,18})> 2.0194$. Setting $H=T_{1,2,18} $, $a=2$ and $\varepsilon = 0.0097$ in Lemma \ref{lms5.1} with the given condition on $n$, we get $S_k(T)\ge 2k$, a contradiction as $S_k(T) < 2k$. Otherwise, $d(v_{2})d(v_{3})\leq  1$, that is, $d(v_2)=d(v_3)=1$. This means that every vertex of degree $3$ is adjacent to exactly two pendent vertices. Hence $T\cong W_n$ or $T\cong Y_n$. This completes the proof.
\end{proof}

By combining Lemmas \ref{lms6} and \ref{lms7}, we obtain the complete
characterization of connected graphs with $S_k(G)<2k$.
 \begin{theorem} \label{lms8} Let $G$ be a connected graph with $n \geq (2k+1)^{20+23(k-1)r_{0.0097}}$. Then $S_k(G)< 2k$  if and only if $G\in \{Y_n,\,W_n,\,P_n,\,C_n\}$.
\end{theorem}  

  It remains to compare the four graph families in
Theorem~\ref{lms8}.
 \begin{lemma} \label{lms9} For $n \geq 5k$, we have $S_{k}(Y_{n})>S_{k}(P_{n})$ and $S_{k}(W_{n})>S_{k}(P_{n})$.
\end{lemma} 

 \begin{proof} Let ${\bf x_1}=(x_{11}, \dots, x_{1n})^T, {\bf x_2}=(x_{21}, \dots, x_{2n})^T, \dots, {\bf x_k}=(x_{k1}, \dots, x_{kn})^T$ be the eigenvectors corresponding to eigenvalues $\lambda_{1}(P_{n}), \lambda_{2}(P_{n}), \dots, \lambda_{k}(P_{n})$ of $P_{n}$, respectively. By Lemma \ref{lm2}, we have $x_{ji} = \sin\left(\frac{ij\pi}{n+1}\right)$ for $1 \leq j \leq k, 1 \leq i \leq n$. Note that ${\bf x_1} \perp {\bf x_2} \perp \dots \perp {\bf x_k}$. We will first prove \(S_{k}(Y_{n})>S_{k}(P_{n})\). By Lemma \ref{lm1}, we have
 \begin{align*}
 S_k(Y_n) & \ge \sum_{j=1}^k \frac{{\bf x_j}^T A(Y_n) {\bf x_j}}{{\bf x_j}^T {\bf x_j}} \\
          & = \sum_{j=1}^k \frac{{\bf x_j}^T A(P_n) {\bf x_j} + 2 x_{j1}(x_{j3}-x_{j2})}{{\bf x_j}^T  {\bf x_j}} \\
          & = \sum_{j=1}^k \frac{{\bf x_j}^T A(P_n) {\bf x_j}}{{\bf x_j}^T {\bf x_j}} + \sum_{j=1}^k \frac{2 x_{j1}(x_{j3}-x_{j2})}{{\bf x_j}^T {\bf x_j}} \\
          & = S_k(P_n) + \sum_{j=1}^k \frac{2 \sin\left(\frac{j\pi}{n+1}\right)\left(\sin\left(\frac{3j\pi}{n+1}\right)-\sin\left(\frac{2j\pi}{n+1}\right)\right)}{{\bf x_j}^T {\bf x_j}} \\
          & > S_k(P_n)
 \end{align*}
as $n \ge 5k$ and for $j=1,\,2,\ldots,\,k$,
$$\sin\left(\frac{3j\pi}{n+1}\right)>\sin\left(\frac{2j\pi}{n+1}\right).$$

 By a similar way, we also have
 \begin{align*}
 S_{k}(W_n) & \geq \sum_{j=1}^k\frac{{\bf x_j}^{T} A\left(W_{n}\right) {\bf x_j}}{{\bf x_j}^{T} {\bf x_j}} \\
            & = \sum_{j=1}^k\frac{{\bf x_j}^{T} A\left(P_{n}\right) {\bf x_j}+2 x_{j1}\left(x_{j3}-x_{j2}\right)+2 x_{jn}\left(x_{j(n-2)}-x_{j(n-1)}\right)}{{\bf x_j}^{T} {\bf x_j}} \\
            & = S_{k}\left(P_{n}\right) + \sum_{j=1}^k\frac{2 x_{j1}\left(x_{j3}-x_{j2}\right)+2 x_{jn}\left(x_{j(n-2)}-x_{j(n-1)}\right)}{{\bf x_j}^{T} {\bf x_j}}.
 \end{align*}
Note that $x_{jn}\left(x_{j(n-2)}-x_{j(n-1)}\right)=x_{j1}\left(x_{j3}-x_{j2}\right)>0$ for $1 \leq j \leq k$ and $n \geq 5k$. Combining the above results, we obtain that $S_{k}(W_{n})>S_{k}(P_{n})$.
\end{proof}

 \begin{lemma} \label{lms10}  For $n \geq \frac{\pi^2}{2}k^2$, we have $S_k(C_n) > S_k(P_n)$.
 \end{lemma}
 
 \begin{proof} We now divide the proof into two cases:

 \vspace*{3mm}
  \noindent
 ${\bf Case\,1.}$ $k$ is odd. Then $k\geq 3$ and 
 $$S_k(C_n) = 2 + 2\sum_{i=1}^{\frac{k-1}{2}} 2\cos\frac{2i\pi}{n}, \quad S_k(P_n) = 2\cos\frac{\pi}{n+1} + 2\sum_{i=1}^{\frac{k-1}{2}}\left(\cos\frac{2i\pi}{n+1} + \cos\frac{(2i+1)\pi}{n+1}\right).$$

 \vspace*{3mm}  
  \noindent
${\bf Claim\,1.}$ $2\cos\frac{2i\pi}{n} > \cos\frac{2i\pi}{n+1} + \cos\frac{(2i+1)\pi}{n+1}$ for all $i = 1,2,\dots,\frac{k-1}{2}$.

 \vspace*{3mm} 
  \noindent\textbf{Proof of Claim 1.} Since $y = \cos x$ is a strictly decreasing and concave function on $x \in \left[0,\frac{\pi}{2}\right]$, it follows that
 $$\cos\frac{2i\pi}{n+1} + \cos\frac{(2i+1)\pi}{n+1} < 2\cos\left(\frac{2i+\frac{1}{2}}{n+1}\pi\right).$$
 Moreover,
 $$2\cos\frac{2i\pi}{n} > 2\cos\left(\frac{2i+\frac{1}{2}}{n+1}\pi\right) \iff \frac{2i\pi}{n} < \frac{2i+\frac{1}{2}}{n+1}\pi \iff 2i(n+1) < 2in + \frac{1}{2}n \iff n > 4i.$$
 Thus, this inequality holds uniformly for \(n > 2k-2\), and Claim 1 is proved.\\
 By \textbf{Claim 1}, we have
 $$2\sum_{i=1}^{\frac{k-1}{2}} 2\cos\frac{2i\pi}{n} > 2\sum_{i=1}^{\frac{k-1}{2}}\left(\cos\frac{2i\pi}{n+1} + \cos\frac{(2i+1)\pi}{n+1}\right),$$
 and since $2 > 2\cos\frac{\pi}{n+1}$, it follows that $S_k(C_n) > S_k(P_n)$.

 \vspace*{3mm}
  \noindent
 ${\bf Case\,2.}$ $k$ is even. Then
 $$S_k(C_n) = 2\left(1+\cos\frac{k\pi}{n}\right) + 2\sum_{i=1}^{\frac{k}{2}-1} 2\cos\frac{2i\pi}{n},~~\mbox{and} $$
 $$ S_k(P_n) = 2\left(\cos\frac{\pi}{n+1} + \cos\frac{k\pi}{n+1}\right) + 2\sum_{i=1}^{\frac{k}{2}-1}\left(\cos\frac{2i\pi}{n+1} + \cos\frac{(2i+1)\pi}{n+1}\right).$$

 \vspace*{3mm}  
  \noindent
 ${\bf Claim\,2.}$ $1 + \cos\frac{k\pi}{n} > \cos\frac{\pi}{n+1} + \cos\frac{k\pi}{n+1}$.

 \vspace*{3mm} 
  \noindent\textbf{Proof of Claim 2.} Set $x=\pi/(n+1)$. Since $n\ge 20$ and $0<x<1$,  Lemma \ref{lm5} gives
$1-\cos x\ge x^2/2-x^4/24>x^2/3$. On the other hand,
\begin{align*}
 \cos(kx)-\cos\frac{k\pi}{n}
 &=\int_{kx}^{k\pi/n}\sin t\,dt\\
 &\leq \int_{kx}^{k\pi/n}t\,dt\\
 &\le\frac{k^2\pi^2}{n^2(n+1)}
 =x^2\frac{k^2(n+1)}{n^2}
 \le x^2\frac{21}{10\pi^2}<\frac{x^2}{3},
\end{align*}
 which completes the proof of \textbf{Claim 2}.
 
 \noindent
 By \textbf{Claims 1} and \textbf{2}, we obtain $S_k(C_n) > S_k(P_n)$.
\end{proof}

 By combining Theorem \ref{lms8} with Lemmas \ref{lms9} and \ref{lms10}, we obtain the main result of this section.
 \begin{theorem}\label{lms11} Let $G$ be a connected graph with $n \geq (2k+1)^{20+23(k-1)r_{0.0097}}$. Then $S_k(G)\geq S_k(P_n)$ with equality holding if and only if $G\cong P_n$.
\end{theorem}

\begin{remark}
{\rm A conjecture concerning the minimum of $S_2(G)$ over connected graphs was proposed in \cite[Conjecture 5.2]{KLMPT} and was recently proved by us in \cite{SMD1}. Theorem~\ref{lms11} extends this result from $S_2(G)$
to $S_k(G)$ for general $k$.}
\end{remark}

The preceding theorem also yields the following weighted version.
\begin{theorem}\label{suncor1}
 let $G$ be connected graph of order $n\ge (2k+1)^{20+23(k-1)r_{0.0097}}$, and let
$a_1\ge a_2\ge\cdots\ge a_k\ge0$ with $a_1>0$. Then
$$\sum_{i=1}^{k}a_i\lambda_i(P_n)\le\sum_{i=1}^{k}a_i\lambda_i(G).$$
Moreover, equality holds if and only if $G\cong P_n$.
\end{theorem}

\begin{proof}
Let $a_{k+1}=0$ and define
\[
b_i=a_i-a_{i+1},
\qquad
i=1,\ldots,k.
\]
Then $b_i\geq 0$ for all $i$. Since $a_{k+1}=0$, we obtain
\begin{align*}
\sum_{i=1}^k\,b_i\,S_i(G)&=\sum_{i=1}^k\,(a_i-a_{i+1})\,\sum\limits_{j=1}^i\,\lambda_j(G)\\[2mm]
&=\sum_{i=1}^k\,a_i\,\sum\limits_{j=1}^i\,\lambda_j(G)-\sum_{i=1}^k\,a_{i+1}\,\sum\limits_{j=1}^i\,\lambda_j(G)\\[2mm]
&=\sum_{i=1}^k\,a_i\,\lambda_i(G)+\sum_{i=2}^k\,a_i\,\sum\limits_{j=1}^{i-1}\,\lambda_j(G)-\sum_{i=1}^{k-1}\,a_{i+1}\,\sum\limits_{j=1}^i\,\lambda_j(G).
\end{align*}
Since 
  $$\sum_{i=2}^k\,a_i\,\sum\limits_{j=1}^{i-1}\,\lambda_j(G)=\sum_{i=1}^{k-1}\,a_{i+1}\,\sum\limits_{j=1}^{i}\,\lambda_j(G),$$
from the above, we have 
$$\sum_{i=1}^k\,b_i\,S_i(G)=\sum_{i=1}^k\,a_i\,\lambda_i(G).$$
By Theorem~\ref{lms11}, we have $S_i(P_n)\le S_i(G)$ for every $i=1,\ldots,k$. Thus we obtain
$$\sum_{i=1}^k\,a_i\,\lambda_i(P_n)
=\sum_{i=1}^k\,b_i\,S_i(P_n)
\le \sum_{i=1}^k\,b_i\,S_i(G)
=\sum_{i=1}^k\,a_i\,\lambda_i(G)$$
with equality holding if and only if $b_i\,S_i(G)=b_i\,S_i(P_n)$ for $i=1,\,2,\ldots,\,k$.
Since $a_1>0$ and $a_{k+1}=0$, then  there exists an index $t$
such that $b_t>0$. Thus $S_t(G)=S_t(P_n)$.
Therefore, by Theorem~\ref{lms11}, we conclude that $G\cong P_n$.
\end{proof}

\begin{remark}{\rm The authors of \cite[Problem 7.2(i)]{KMPZ} raised the question of whether the path \(P_n\) minimizes $\alpha\lambda_1(G)+(1-\alpha)\lambda_2(G)$ for $\alpha \in (0.5,1)$ among all trees of order $n$. Theorem \ref{suncor1} resolves this question in a more general form: it establishes the desired inequality for all connected graphs of sufficiently large order.}
\end{remark}

 \section{Characterization of graphs in the first Hoffman-type range for eigenvalue sum}
In the previous section, we characterized the connected graphs in the
subcritical range $S_k(G)<2k$. We now move to the first Hoffman-type
range above the threshold $2k$. Motivated by the classical interval
between $2$ and $\eta=\sqrt{2+\sqrt5}$ for the adjacency spectral radius,
we consider connected graphs satisfying
\[
2k\leq S_k(G)<2k+\eta-2.
\]
 We first consider connected graphs containing a cycle.
 \begin{lemma}\label{smlm1}
 Let $G$ be a connected graph with $m \ge n \geq (2k+1)^{21+25(k-1)r_{0.0009}}$. 
 If $S_k(G)\in [2k,\, 2k + \eta - 2)$, then $G\cong U_{n-1}$.
 \end{lemma}
 
 \begin{proof} We consider the following cases according to the maximum value of $\Delta$.

 \vspace*{3mm}

 \noindent
 ${\bf Case\,1:}$ $\Delta \ge 4$. Similar to the case $\Delta\geq 4$ in the proof of Lemma \ref{lms7},  $G$ contains a subgraph $T'$ of order $6$ with $\lambda_1(T') \geq 2.074$.
 Setting $H=T'$, $a=\eta$ and $\varepsilon = 0.007$ in Lemma \ref{lms5.1}, together with the fact that $$n \ge (2k+1)^{21+25(k-1)r_{0.0009}}> (2k+1)^{4+7(k-1)r_{0.007}},$$
 we get $S_k(G) \ge 2k + \eta - 2$, a contradiction as $S_k(G)\in [2k,\, 2k + \eta - 2)$. 

  \vspace*{3mm}

 \noindent
 ${\bf Case\,2:}$ $\Delta \leq 3$. For $\Delta = 2$, we have $G\cong C_n$, and by Lemma \ref{lms6}, we obtain $S_k(G)<2k$,  a contradiction as $S_k(G)\in [2k,\, 2k + \eta - 2)$. It remains to consider the case $\Delta = 3$. Let $C_t$ be a shortest cycle in $G$. Then there exists a vertex $x \in V(C_t)$ with $d(x) = 3$. So there exists $y \in V(G) \setminus V(C_t)$ such that $xy \in E(G)$. Therefore, $U_t$ is a subgraph of $G$. Now we divide the rest of the proof into two cases based on $t$.

 \vspace*{3mm}
 
  \noindent
  ${\bf Case\,2.1:}$ $3 \le t \le 10$. Since $U_t$ is a subgraph of $G$, we obtain $\lambda_1(U_t) \ge \lambda_1(U_{10}) > 2.074$.
  Setting $H=U_t$, $a=\eta$ and $\varepsilon = 0.007$ in Lemma \ref{lms5.1}, we obtain $S_k(G) \ge 2k + \eta - 2$, a contradiction as $S_k(G)\in [2k,\, 2k + \eta - 2)$.

 \vspace*{3mm}
  \noindent
  ${\bf Case\,2.2:}$ $t \geq 11$. If $G[V(G) \setminus V(C_t)]$ contains at least one edge, then $T_{2,4,4}$ is a subgraph of $G$ and hence $\lambda_1(T_{2,4,4})> 2.074$.
   Setting $H=T_{2,4,4}$, $a=\eta$ and $\varepsilon = 0.007$ in Lemma \ref{lms5.1},
 we obtain $S_k(G) \ge 2k + \eta - 2$, a contradiction as $S_k(G)\in [2k,\, 2k + \eta - 2)$. Otherwise, $G[V(G) \setminus V(C_t)]$ contains only isolated vertices, which means that $N_G(v)\subset{V(C_t)}$ for any $v\in V(G)\setminus V(C_t)$. As $C_t$ is defined as a shortest cycle in $G$, then $d_G(v)=1$ for any $v\in V(G)\setminus V(C_t)$.
 Therefore, $G$ is a closed quipu and all vertices having degree at least two are in the cycle $C_t$. Then we have $3t\geq n \ge (2k+1)^{21+25(k-1)r_{0.0009}}>90$.
 If there is exactly one vertex of degree $3$ in $C_t$, then $G\cong U_{n-1}$. Otherwise, there exists at least two vertices $u_1, u_2 \in V(C_t)$ such that $d_G(u_1) = d_G(u_2) = 3$. Let $d = d(u_1,\ u_2)$. We will show it is impossible for any $d$ by considering the following two subcases.
  
 \vspace*{3mm}
  \noindent
  ${\bf Case\,2.2.1:} $ $d \le 7$. Since $t>30$, $G$ contains a subgraph $P^{1,\,d+1}_{1,1,20}$. By SAGE \cite{SAGE}, one can easily check that $\lambda_1(P^{1,d+1}_{1,1,20}) \geq \lambda_1(P^{1,\,8}_{1,1,20}) > 2.06$ for $1\leq d\leq 7$. 
  Setting $H=P^{1,\,d+1}_{1,1,20}$, $a = \eta$ and $\varepsilon = 0.0009$ in Lemma \ref{lms5.1}, 
  we get $S_k(G) \geq 2k + \eta - 2$, a contradiction as $S_k(G)\in [2k,\, 2k + \eta - 2)$.
  
 \vspace*{3mm}
  \noindent
 ${\bf Case\,2.2.2:}$ $d \geq 8$. As $t>30$, there exist two vertex-disjoint copies of $T_{1,3,7}$ rooted at $u_1$ and $u_2$, respectively, and $\lambda_1(T_{1,3,7})> 2.03$.
 Setting $H_1=H_2=T_{1,3,7}$, $b =2+ \eta$ and $\varepsilon = 0.0009$ in Lemma \ref{lms5.2} , together with the fact that $$n \ge (2k+1)^{21+25(k-1)r_{0.0009}}> (2k+1)^{21+25(k-2)\theta_{0.0009}},$$ 
 we have $S_k(G) \ge 2k + \eta - 2$,  again a contradiction as $S_k(G)\in [2k,\, 2k + \eta - 2)$. This completes the proof.
\end{proof}

 \begin{theorem}\label{smlm6}
 Let $G$ be a connected graph with $m \ge n$. For sufficiently large $n$ relative to $k$,  we have $S_k(G)\in [2k,\, 2k + \eta - 2)$ if and only if $G \cong U_{n-1}$.
 \end{theorem}

 \begin{proof}
 The necessity follows from Lemma~\ref{smlm1}. It remains to prove
the sufficiency. For convenience, we take $U_n$ of order $n+1$.

 \vspace*{3mm}
   \noindent
   ${\bf Step\,1.}$ Determine the $\lim_{n\to\infty}\lambda_1(U_n)$.
   As $n$ is sufficiently large, then $T_{1,5,5}$ is a subgraph of $U_n$. Thus $\lambda_1(U_n) \ge \lambda_1(T_{1,5,5})>2$ by Lemma \ref{lm3}.  Moreover, by Lemma \ref{lm8}, we have $\lambda_1(U_n) >\lambda_1(U_{n+1})$.
    These results imply that  $\lim_{n\to\infty}\lambda_1(U_n)$ exists, say $\lambda^*$.
   Let $V(U_n) = \{u,\,v_0, v_1, v_2, \dots, v_{n-1}\}$ such that $V(C_n)=V(U_n)\setminus \{u\}$ and $uv_0 \in E(U_n)$. Assume that ${\bf x}=(y, x_0, x_1, \dots, x_{n-1})^T$ is the eigenvector corresponding to $\lambda=\lambda_1(U_n)$. 
   From $A(U_n){\bf x}=\lambda {\bf x}$, we have   
 \begin{equation}
 \begin{cases}\label{sys11}
  \lambda y = x_0, \\
 \lambda x_0 = x_1 + x_{n-1} + y, \\
  \lambda x_i = x_{i-1} + x_{i+1}, ~~i=1,\,2,\ldots,n-2, \\
 \lambda x_{n-1} = x_0 + x_{n-2}.
 \end{cases}
\end{equation}
Given the recurrence relation 
 $\lambda x_i = x_{i-1} + x_{i+1}$ for $1 \le i \le n-2$, its characteristic equation is
  $x^2 - \lambda x + 1 = 0$. As $\lambda>2$, the general solution is
 $x_i = A\alpha^i + B\beta^i$ for $i=0,\,1,\ldots,n-1$, 
 where
 $\alpha = \frac{\lambda + \sqrt{\lambda^2 - 4}}{2} > 1, \quad 0 < \beta = \frac{\lambda - \sqrt{\lambda^2 - 4}}{2}=\frac{2}{\lambda + \sqrt{\lambda^2 - 4}}< 1.$
Thus we have $\alpha\,\beta=1$ and $\alpha+\beta=\lambda$, where $\alpha>1$ and $0<\beta<1$. Due to the above fact, the system (\ref{sys11}) is equivalent to the following one:
$$
  \begin{cases}\label{sys1}
 \lambda x_0 = x_1 + x_{n-1} + \frac{x_0}{\lambda}, \\[1mm]
   \lambda x_{n-1} = x_0 + x_{n-2}, \\[1mm]
   x_i =A\alpha^i + B\beta^i, ~~i=0,\,1,\ldots,n-1.
 \end{cases}
 $$
From the above, we obtain
\begin{align*}
\lambda\,(A\alpha^{n-1} + B\beta^{n-1})=A+B+A\alpha^{n-2} + B\beta^{n-2},~\mbox{ that is, }~A\,(\alpha^n-1) + B\,(\beta^n-1)=0
\end{align*}
as $\lambda=\alpha+\beta$ and $\alpha\,\beta=1$. Moreover, 
\begin{align*}
&\lambda\,(A+B)=A\alpha+B\beta+A\alpha^{n-1} + B\beta^{n-1}+\frac{A+B}{\lambda},\\[2mm]
\mbox{ that is, }&~A\left(\alpha + \alpha^{n-1} - \lambda + \frac{1}{\lambda}\right) + B\left(\beta + \beta^{n-1} - \lambda + \frac{1}{\lambda}\right) = 0.
\end{align*}
 Note that $(A,\,B)\neq (0,\,0)$ (otherwise, $x_i=0$ for every $i$, and then ${\bf x}$ is zero vector, a contradiction).
 Then the above linear system has nonzero solutions, which implies that 
 $$
 \begin{vmatrix}
 \alpha^n - 1 & \beta^n - 1 \\[2mm]
 \alpha + \alpha^{n-1} - \lambda + \frac{1}{\lambda} & \beta + \beta^{n-1} - \lambda + \frac{1}{\lambda}
 \end{vmatrix}
 = 0.
 $$
 Expanding the determinant yields
 $$(\alpha^n - 1)\,\Big(\beta + \beta^{n-1} - \lambda + \frac{1}{\lambda}\Big)-(\beta^n - 1)\,\Big(\alpha + \alpha^{n-1} - \lambda + \frac{1}{\lambda}\Big)=0.$$
 Since $\alpha\,\beta=1$, from the above, we obtain
 $$\left(\frac{1}{\beta^n}-1\right)\,\Big(\beta + \beta^{n-1} - \lambda + \frac{1}{\lambda}\Big)-(\beta^n - 1)\,\left(\frac{1}{\beta}+\frac{1}{\beta^{n-1}}- \lambda + \frac{1}{\lambda}\right)=0,$$
 that is,
 $$\left(\frac{1}{\beta^n}-1\right)\,\left(2\beta + 2\beta^{n-1} - \Big(\lambda - \frac{1}{\lambda}\Big)\,(\beta^n+1)\right)=0.$$
Since $\beta<1$, from the above, we have
 \begin{equation}
2\beta + 2\beta^{n-1} - \Big(\lambda-\frac{1}{\lambda}\Big)\,(\beta^n+1)=0.\label{sss1}
\end{equation}
 Taking the limit on both sides as $n \to \infty$, we have 
 $$2\,\lim_{n \to \infty}\,\beta-\left(\lambda^*-\frac{1}{\lambda^*}\right)=0,~\mbox{ that is, }~
 \left(\lambda^* - \frac{1}{\lambda^*}\right)\left(\frac{\lambda^* + \sqrt{(\lambda^*)^2 - 4}}{2}\right) = 2
 $$
 as
 $$\lim_{n \to \infty}\,\lambda=\lambda^*,\,\lim_{n \to \infty}\,\beta=\frac{2}{\lambda^* + \sqrt{(\lambda^*)^2 - 4}},~\lim_{n \to \infty}\,\beta^n=\lim_{n \to \infty}\,\beta^{n-1}=0.$$
 From the above equation, we get the unique root  $\eta$, which is greater than $2$.
 Hence $\lambda^*=\eta$.

\vspace*{3mm}
   \noindent
   ${\bf Step\,2.}$ Estimate the rate of convergence of $\lambda_1(U_n)$.
Let $\beta^*=\beta_{\lambda^*}$, where $\beta_{\lambda}=\beta$ is the decreasing function of $\lambda$ defined in \textbf{Step 1}. Recall that $\lambda=\lambda_1(U_n)$ decreases as $n$ increases and $\lim_{n\to\infty}\lambda_1(U_n)=\lambda^*$. Then $\lambda\in (\lambda^*,\,2.17)$ for sufficiently large $n$. This implies that $\beta <\beta^*$. Note that for  $\lambda\in (2,2.17)$, we have 
$$\sqrt{\lambda^2-4}+\frac{2}{\sqrt{\lambda^2-4}}>\sqrt{0.7089}+\frac{2}{\sqrt{0.7089}}>3.2~\mbox{and }\lambda+\frac{1}{\lambda^2}<2.17+0.25=2.42.$$
Then
$$(2-\lambda\beta+\lambda^{-1})^\prime=\sqrt{\lambda^2-4}+\frac{2}{\sqrt{\lambda^2-4}}-\lambda-\frac{1}{\lambda^2}>0,$$
which implies that $2-\lambda\beta+\lambda^{-1}$ is increasing on $\lambda\in (2,2.17)$.
From (\ref{sss1}) with the above fact, we have
  \begin{eqnarray}
  \lambda&= &\beta^{n-1}(2-\lambda\beta+\lambda^{-1}\beta)+\lambda^{-1}+2\beta\nonumber\\
  &<&\beta^{n-1}(2-\lambda\beta+\lambda^{-1})+\lambda^{-1}+2\beta\nonumber\\
        &<&\beta^{n-1}(2-2.17\beta_{2.17}+2.17^{-1})+\lambda^{-1}+2\beta\nonumber\\
        &<&(\beta^*)^{n-1}(2-2.17\beta_{2.17}+2.17^{-1})+(\lambda^*)^{-1}+2\beta^*~~\mbox{as }\lambda>\lambda^*,\,\beta<\beta^*\nonumber\\
        &=&(\beta^*)^{n-1}(2-2.17\beta_{2.17}+2.17^{-1})+\lambda^*.\nonumber
 \end{eqnarray}
Thus we conclude that $\lambda=\lambda^*+O((\beta^*)^n)$, where $\beta^*=\frac{\sqrt{2+\sqrt5}-\sqrt{\sqrt5-2}}{2}$.

\vspace*{3mm}
   \noindent
   ${\bf Step\,3.}$ Prove $S_k(U_n)<2k + \eta - 2$.
Since $U_n$ contains $P_{n-1}\cup K_1$ as an induced subgraph, by Lemmas  \ref{lm5} and \ref{lm6}, we get 
$$\lambda_2(U_n)\leq \lambda_1(P_{n-1})=2\cos\left(\frac{\pi}{n}\right)\leq 2-\frac{\pi^2}{n^2}+\frac{\pi^4}{12n^4},$$
  which implies that $\lambda_i(U_n)<2$ for $i\geq 3$ and there exists $c_1 > 0$ such that $\lambda_2< 2-\frac{c_1}{n^2}$. Combining the above results with the result in \textbf{Step 2}, there exists $c_2$ such that  
  $$S_k(U_n)< \eta+c_2(\beta^*)^{n}+2-c_1/n^2+2(k-2)<2k+\eta-2$$
  as $c_2(\beta^*)^{n}-c_1/n^2<0$ for sufficiently large $n$.
  Moreover, by Lemma \ref{lms6}, we obtain $S_k(U_n)\in [2k,\,2k + \eta - 2)$.
  This completes the proof.
\end{proof}

We next turn to the tree case. Unlike the non-tree case, the graphs
in the target spectral range are not determined uniquely. We show
that every such tree must belong to a small number of explicit families.
\begin{theorem}\label{smlm3}
Let $T$ be a tree with $n \ge (2k+1)^{30+33(k-1)r_{0.00006}}$.
If $S_k(T) \in [2k,\,2k + \eta - 2)$, then $T\in \mathcal{T}$, where
$$\mathcal{T}= \left\{ P_{1,\,1,\,1,\,n-3}^{1,\,a,\,n-5},\, P_{1,\,1,\,n-2}^{1,\,b},\,P_{1,\,1,\,n-2}^{2,\,n-6},\, P_{1,\,1,\,n-2}^{2,\,n-5},\,  P_{1,\,2,\,n-3}^{1,\,n-6},\, T_{2,2,n-5},\,T_{1,c,n-c-2} \right\}$$
with $a\in [4,\,n-8]$, $b\in [4,\,n-5]$ and $c\in [2,\,n-4]$.
\end{theorem}

\begin{proof} If $\Delta \ge 4$, the argument at the beginning of the proof of Lemma~\ref{smlm1} gives $S_k(T) \ge 2k + \eta - 2$, a contradiction. Hence $\Delta\leq3$.
If $\Delta=2$, then $T\cong P_n$, and hence $S_k(T)<2k$.
Therefore, it remains to consider $\Delta=3$.
 Let $P_s$ be a longest path in $T$ with $v_iv_{i+1}\in E(P_s)$ for $i=0,\,1,\ldots,\,s-2$.
   As $n \geq (2k+1)^{30+33(k-1)r_{0.00006}}$, by Lemma \ref{lms4}, we obtain 
    $$s>2\log_3n-2> 30+33(k-1)r_{0.00006}>63.$$
  \noindent
  ${\bf Claim\,3.}$ Every degree-3 vertex lies on $P_s$; in particular, $T$ is an open quipu.
  
  \vspace*{3mm} 
  \noindent\textbf{Proof of Claim 3.} 
  Suppose, to the contrary, that there exists a vertex
$u\in V(T)\setminus V(P_s)$ with $d_T(u)=3$. Let
$v_q\in V(P_s)$ be chosen so that
$$d_T(v_q,u)=\min_{0\leq i\leq s-1}d_T(v_i,u).$$
  As $P_s$ is a longest path and $d_T(u)=3$, then $2\leq q\leq s-3$.
  If $q=2$ or $q=s-3$, then $d(v_q, u) = 1$ (otherwise, it contradicts that $P_s$ is a longest path). Then $T$ contains a subgraph $T_1$ of order $8$, which is obtained from $T_{2,2,2}$ by adding a pendant edge to the vertex of degree $2$. Note that $\lambda_1(T_1)>2.101$.
 Setting $H=T_1$, $a = \eta$ and $\varepsilon = 0.021$ in Lemma \ref{lms5.1}, 
 we have $S_k(T) \geq 2k + \eta - 2$, is a contradiction. Otherwise, $3\leq q\leq s-4$. Then $T_{2,3,4}$ is a subgraph of $T$ and $\lambda_1(T_{2,3,4})> 2.064$.
Setting $H=T_{2,3,4}$, $a = \eta$ and $\varepsilon = 0.002$ in Lemma \ref{lms5.1}, we again get a contradiction. Therefore, for any $v \in V(T) \setminus V(P_s)$, we have $d_T(v) \le 2$, and thus $T$ is an open quipu.

  \vspace*{3mm} 
  By \textbf{Claim 3}, we assume that $T\cong P_{n_1,\,n_2,\ldots,\,n_t,\,s}^{m_1,\,m_2,\ldots,\,m_t}$.

  \noindent
  ${\bf Claim\,4.}$ $t\leq 3$. Moreover, each of the sets $\{v_1, v_2\}$, $\{v_{s-3}, v_{s-2}\}$  and $\{v_{3},v_4,\ldots, v_{s-4}\}$ contains at most one vertex of degree $3$.
  
  \vspace*{3mm} 
  \noindent\textbf{Proof of Claim 4.} 
 Let $v_i, v_j\in V(P_s)$ be be two vertices of degree $3$ with $3\leq i<j\leq s-4$. If $1 \le j-i \le 14$, then $P_{1,\,1,\,30}^{3,\,j-i+3}$ is a subgraph of $T$, where $\lambda_1(P_{1,\,1,\,30}^{3,\,j-i+3})\geq\lambda_1(P_{1,\,1,\,30}^{3,\,17}) \ge 2.0585$.
 Setting $H=P_{1,\,1,\,30}^{3,\,j-i+3}$, $a = \eta$ and $\varepsilon = 0.00015$ in Lemma \ref{lms5.1}, we get a contradiction. Otherwise, $j-i\geq 15$. Then there exist two vertex-disjoint copies of $T_{1,3,7}$, with $2\lambda_1(T_{1,3,7}) >4.0628$.
 Setting $H_1=H_2=T_{1,3,7}$, $b = 2+\eta$ and $\varepsilon = 0.0023$ in Lemma \ref{lms5.2}, we get a contradiction. 
 Thus $\{v_3,\ldots,v_{s-4}\}$ contains at most one vertex of degree
$3$. 
 Furthermore, no two degree-$3$ vertices on $P_s$ can be adjacent (otherwise, $T$ contains a subgraph $P_{1,\,1,\,6}^{1,\,2}$ with  $\lambda_1(P_{1,\,1,\,6}^{1,\,2})>2.074$, a contradiction by Lemma \ref{lms5.1})). Consequently, each of $\{v_1,v_2\}$ and
$\{v_{s-3},v_{s-2}\}$ contains at most one vertex of degree $3$.
This proves the claim.

  \vspace*{3mm} 
  \noindent
  ${\bf Claim\,5.}$  Suppose there exists $i\in [3,\,s-4]$ such that $d_T(v_i) = 3$. Then $n_i=1$. Moreover, $n_{s-3}\leq 1$ ($n_{2}\leq 1$) if $i=3$ ($i=s-4$); otherwise, $n_{s-3}=0$ ($n_{2}=0$).
  
  \vspace*{3mm} 
  \noindent\textbf{Proof of Claim 5.} 
  Due to the given condition, we conclude that $v_i$ is adjacent to a pendant vertex (otherwise, $T$ contains a subgraph $T_{2,\,3,\,4}$ with  $\lambda_1(T_{2,\,3,\,4})>2.064$, a contradiction by Lemma \ref{lms5.1}). Thus we have $n_i=1$. 
  Now we will prove $n_{s-3}\leq 1$ if $i=3$ and $n_{s-3}=0$ if $i\in [4,\,s-4]$ by contradiction. For this, we assume that  $n_{s-3}\geq  2$ if $i=3$ and $n_{s-3}\geq 1$ if $i\in [4,\,s-4]$. For $i=3$, $T_{1,3,6} \cup T_{2,2,7}$
 is a subgraph of $T$ as  $n_{s-3}\geq  2$, and hence  $\lambda_1(T_{1,3,6})> 2.028$, $\lambda_1(T_{2,2,7})>2.055$.  Setting $H_1=T_{1,3,6}$, $H_2=T_{2,2,7}$, $b = 2+\eta$ and $\varepsilon = 0.007$ in Lemma \ref{lms5.2}, we get a contradiction. For $i\in [4,\,s-18]$, $T_{1,2,10} \cup T_{1,4,11}$
 is a subgraph of $T$, and hence  $\lambda_1(T_{1,2,10})> 2.016$ and $\lambda_1(T_{1,4,11})>2.045$. Setting $H_1=T_{1,2,10}$, $H_2=T_{1,4,11}$, $b = 2+\eta$ and $\varepsilon = 0.0014$ in Lemma \ref{lms5.2}, we get a contradiction.
 Otherwise, $i \in [s-17,\,s-4]$. Then $P_{1,\,1,\,30}^{2,\,s-i-1}$ is a subgraph of $T$, where $\lambda_1(P_{1,\,1,\,30}^{2,\,s-i-1})\geq\lambda_1(P_{1,\,1,\,30}^{2,\,16}) >2.0583$, a contradiction by Lemma \ref{lms5.1} (take $a=\eta$ and $\epsilon=0.00006$). Hence we conclude that $n_{s-3}\leq 1$ if $i=3$ and $n_{s-3}=0$ if $i\in [4,\,s-4]$. By symmetry, we also obtain  $n_{2}\leq 1$ if $i=s-4$,  $n_{2}=0$ otherwise.
 
\vspace*{3mm} 
By ${\bf Claim\,4}$, we have $t\leq 3$.  We divide into the following three cases according to $t$. 

  \vspace*{3mm} 
  \noindent\textbf{Case 1.}  $t=3$. Then $T\cong P_{n_1,\,n_2,\,n_3,\,s}^{m_1,\,m_2,\,m_3}$. By {\bf Claims  4} and {\bf 5}, we have $n_2=1$.
  If $m_2=3$ or $m_2=s-4$, then $T$ contains $P_{1,\,1,\,7}^{1,\,3}$ or $P_{1,\,1,\,7}^{2,\,3}$ as a subgraph. Since $\lambda_1( P_{1,\,1,\,7}^{2,\,3})>\lambda_1( P_{1,\,1,\,7}^{1,\,3})>2.06$, then by Lemma \ref{lms5.1}, we get a contradiction. Otherwise,  $m_2\in [4,\,s-5]$. Then by {\bf Claim 5} with the fact that $P_s$ is the longest path, we have $T\cong P_{1,\,1,\,1,\,n-3}^{1,\,a,\,n-5}$ with $a\in [4,\,n-8]$.

  \vspace*{3mm} 
  \noindent\textbf{Case 2.}  $t=2$. Without loss of generality, we assume that $m_1\in[1,\,2]$. Then $m_2\geq 4$ as the subgraphs $P_{1,\,1,\,7}^{1,\,3}$ and $P_{1,\,1,\,7}^{2,\,3}$ are forbidden from the previous case. For $m_2\in[4,\,s-5]$, then by {\bf Claim 5}, we have $m_1=n_1=n_2=1$, that is, 
  $T\cong P_{1,\,1,\,n-2}^{1,\,b}$ with $b\in [4,\,n-7]$. For $m_2=s-4$, then by {\bf Claim 5}, we have $T\in \{P_{1,\,1,\,n-2}^{1,\,n-6},\,P_{1,\,1,\,n-2}^{2,\,n-6}\}$. For $m_2\in [s-3,\,s-2]$, then 
  $$T\in \{ P_{1,\,1,\,n-2}^{1,\,n-4},\, P_{1,\,1,\,n-2}^{2,\,n-4},\, P_{1,\,1,\,n-2}^{2,\,n-5},\,  P_{1,\,2,\,n-3}^{1,\,n-6},\,  P_{1,\,2,\,n-3}^{2,\,n-6},\,  P_{2,\,2,\,n-4}^{2,\,n-7}\}.$$
  By Lemma \ref{lms7}, we have $S_k(P_{1,\,1,\,n-2}^{1,\,n-4})<2k$. Thus we have $T\ncong P_{1,\,1,\,n-2}^{1,\,n-4}$. If $T\in\{ P_{2,\,2,\,n-3}^{1,\,n-6},\, P_{2,\,2,\,n-4}^{2,\,n-7}\}$, then $T$ contains $T_{1,2,9}\cup T_{2,2,7}$ as a subgraph, and hence $\lambda_1(T_{1,2,9})>2.015$ and $\lambda_1(T_{2,2,7})>2.055$. Setting $H_1=T_{1,2,9}$, $H_2=T_{2,2,7}$, $b = 2+\eta$ and $\varepsilon = 0.0009$ in Lemma \ref{lms5.2}, we get a contradiction. Hence $T\in \{  P_{1,\,1,\,n-2}^{2,\,n-4},\, P_{1,\,1,\,n-2}^{2,\,n-5},\,  P_{1,\,2,\,n-3}^{1,\,n-6}\}$. Note that $P_{1,\,1,\,n-2}^{2,\,n-4}\cong P_{1,\,1,\,n-2}^{1,\,n-5}$.
  
  \vspace*{3mm} 
  \noindent\textbf{Case 3.}  $t=1$. By {\bf Claim 5} with the fact that $P_s$ it the longest path in $T$, then $n_1=1$ if $m_1\notin\{3,\,n-5\}$, and $n_1\leq 2$ otherwise. Then $T\cong T_{2,2,n-5}$ or $T\cong T_{1,c,n-c-2}$ with $c\in [1,\,n-2]$. As $S_k(T_{1,1,n-3})<2k$ by Lemma \ref{lms7}, we have $c\in [2,\,n-4]$. 
  
    \vspace*{3mm} 

    To check all order conditions, the one-subgraph applications above have
$q\le32$ and $\varepsilon\ge 0.00006$. The two-subgraph applications have
$q_1+q_2\le31$ and the same lower bound on $\varepsilon$.
Their exponents are bounded respectively by
\[
 30+33(k-1)r_{0.00006}
 \quad\text{and}\quad
 28+32(k-2)\theta_{0.00006},
\]
and the latter is at most the former.  This completes the proof.

\end{proof}

\begin{remark}\label{sunremark1}{\rm 
The converse of Theorem~\ref{smlm3} does not hold in general; that
is, not every tree in $\mathcal{T}$ satisfies
\[
2k\leq S_k(T)<2k+\eta-2.
\]
Indeed, consider the family
$P_{1,1,n-2}^{1,b}$. Choose $b$ such that
$
n\geq
(2k+1)^{2b+1+(2b+4)(k-1)r_{\varepsilon'}},$
where
\[
2\varepsilon'
=
\lambda_1\left(P_{1,1,2b+1}^{1,b}\right)
-\eta.
\]
By Lemmas~\ref{lm4} and~\ref{lm9}, we have $\varepsilon'>0$.
Applying Lemma~\ref{lms5.1} with
\[
H=P_{1,1,2b+1}^{1,b},\qquad
a=\eta,\qquad
\varepsilon=\varepsilon',
\]
gives
\[
S_k\left(P_{1,1,n-2}^{1,b}\right)
>
2k+\eta-2.
\]
Moreover, since
$P_{1,1,1,n-3}^{1,b,n-5}$ contains
$P_{1,1,n-3}^{1,b}$ as an induced subgraph, the same choice of $b$
also gives
\[
S_k\left(P_{1,1,1,n-3}^{1,b,n-5}\right)
>
2k+\eta-2.
\]}
\end{remark}

\section{Concluding remarks}

Motivated by the classical Hoffman program for the adjacency spectral
radius, in this paper we study an additive Hoffman-type problem for
the sum of the $k$ largest eigenvalues of graphs. The classical
spectral-radius threshold $2$ naturally gives rise to the additive
threshold $2k$. We completely characterize the connected graphs $G$
of sufficiently large order satisfying $S_k(G)<2k.$
As a consequence, we prove that the path is the unique graph
minimizing $S_k(G)$ among all connected $n$-vertex graphs,
extending the corresponding result for $S_2(G)$ to general $k$.

We then consider the next range suggested by the classical Hoffman
theory. The interval between $2$ and $\sqrt{2+\sqrt5}$ for the
spectral radius leads naturally to the additive spectral range
\[
2k\leq S_k(G)<2k+\sqrt{2+\sqrt5}-2.
\]
For non-tree graphs, we obtain a complete characterization. For
trees, we show that every graph in this range must belong to the
explicit family $\mathcal{T}$ given in Theorem~\ref{smlm3}. Thus, the remaining part of this Hoffman-type classification problem is reduced to determining precisely which trees of $\mathcal{T}$ lie in the above spectral range.
This leads to the following problem.
\begin{problem}\label{sunprob1}
Determine all trees \(T\) from the family \(\mathcal{T}\) such that
$S_k(T)\in \left[2k,\; 2k+\eta - 2\right)$
for sufficiently large \(n\) relative to \(k\).
\end{problem}

Within $\mathcal{T}$, the cases $T_{2,2,n-5}$ and $T_{1,c,n-c-2}$ for $c\in[2,n-4]$ are easily shown to satisfy $S_k(T)\in [2k,\,2k+\eta - 2)$ by Lemma \ref{lm9}. The main difficulty lies in determining the precise conditions on the parameters $a$ and $b$ for the two families \(P_{1,1,1,n-3}^{1,a,n-5}\) and \(P_{1,1,n-2}^{1,b}\). 
For the remaining three families in $\mathcal{T}$, we expect that the
required upper bound can be obtained by arguments similar to those
used in the proof of Theorem~\ref{smlm6}. A solution of
Problem~\ref{sunprob1} would therefore complete the classification of
trees in this first additive Hoffman-type spectral range.

More generally, the results suggest studying further additive
analogues of the Hoffman program for $S_k(G)$. In particular, it is
natural to ask whether higher spectral-radius thresholds and the
corresponding structural classifications admit analogous counterparts
for $S_k(G)$.

 \vspace*{3mm}
 
 \noindent
 {\bf Acknowledgement.} S. Sun is supported by National Natural Science Foundation of China (Grant No. 12271484). K. C. Das is supported by National Research Foundation funded by the Korean government (Grant No. RS-2026-25475577). We hereby acknowledge that GPT--5.5 Plus was used solely for language checking and improvement.

 \vspace*{3mm}

\noindent
 \section*{Author contributions}
Shaowei Sun: Writing–review \& editing, Writing–original draft, Software, Methodology, Investigation, Conceptualization; Mengyao Guo: Writing–review \& editing, Writing–original draft, Software, Methodology, Investigation, Formal analysis; Hongyan Ge: Writing–review \& editing, Software, Methodology, Formal analysis;
Kinkar Chandra Das: Writing–review \& editing, Software, Methodology, Investigation, Formal analysis.

 \vspace*{3mm}

\noindent
{\bf Declaration of competing interest.} The authors declare that they have no competing financial interests or personal relationships that could have influenced this work.

 \vspace*{3mm}
 
\noindent
{\bf Data availability.} Not Applicable.


\vspace*{4mm}

 \begin{thebibliography}{99}

\bibitem{AH} M. Aouchiche, P. Hansen, A survey of automated conjectures in spectral graph theory, {\it Linear Algebra Appl.} {\bf 432} (9) (2010) 2293--2322.

 \bibitem{BN} A.E. Brouwer, A. Neumaier, The graphs with spectral radius between $2$ and $\sqrt{2+\sqrt{5}}$, {\it Linear Algebra Appl.} {\bf 114/115} (1989) 273--276.

 \bibitem{CDG} D. Cvetkovi\'c, M. Doob, I. Gutman, On graphs whose spectral radius does not exceed $\sqrt{2+\sqrt{5}}$, {\it Ars Comb.} {\bf 14} (1982) 225--239.
 
\bibitem{BOOK} D. Cvetkovi\'c, M. Doob, H. Sachs, {\it Spectra of Graphs-Theory and Application}, Academic Press, New York, III edition, Berth, Heidelberg, 1980.

\bibitem{CDK} S. M. Cioabă, E. R. van Dam, J. H. Koolen, J. Lee, Asymptotic results on the spectral radius and the diameter of graphs, {\it Linear Algebra Appl.} {\bf 432} (2010) 722--737.
\bibitem{DMS} K. C. Das, S. A. Mojallal, S. Sun, On the sum of the $k$ largest eigenvalues of graphs and maximal energy of bipartite graphs, {\it Linear Algebra Appl.} {\bf 569} (2019) 175--194.
\bibitem{EMNA} J.B. Ebrahimi, B. Mohar, V. Nikiforov, A.S. Ahmady, On the sum of two largest eigenvalues of a symmetric matrix,  {\it Linear Algebra Appl.\/} 429 (2008) 2781--2787.

\bibitem{FAN2} K. Fan, On a theorem of Weyl concerning eigenvalues of linear transformations I, {\it Proc. Natl. Acad. Sci. USA} {\bf 35} (1949) 652--655.

 \bibitem{HW} J. Huang,  W. Wei, The exact maximum of the spectral sum of graphs, arXiv:2607.23081 (17 Pages)
 
 \bibitem{Hoffman} A.J. Hoffman, On limit points of spectral radii of non-negative symmetric integral matrices, in: Y. Alavi, et al. (Eds.), {\it Lecture Notes Math.} {\bf 303} (1972) 165--172.
 
 \bibitem{HS} A.J. Hoffman, J.H. Smith, {\it On the spectral radii of topologically equivalent graphs, in: M. Fiedler (Ed.), Recent Advances in Graph Theory}, Academia Praha, 1975, pp. 273--281.

\bibitem{KLMPT} H. Kumar, L. Liu, H. Monterde, S. Pragada, M. Tait, Maximum spectral sum of graphs, arXiv: 2604.00512v2 (26 Pages)
 \bibitem{KMPZ} H. Kumar, B. Mohar, S. Pragada, H. Zhan, Convex combination of first and second eigenvalues of trees, arXiv:2601.10036 (27 Pages)

\bibitem{MOHAR} B. Mohar, On the sum of $k$ largest eigenvalues of graphs and symmetric matrices, {\it J. Combin. Theory Ser. B.} {\bf 99} (2009) 306--313.

\bibitem{NIKIFOROV} V. Nikiforov, Beyond graph energy: norms of graphs and matrices, {\it Linear Algebra Appl.} {\bf 506} (2016) 82--138.

\bibitem{NI} V. Nikiforov, Linear combinations of graph eigenvalues, {\it Electron. J. Linear Algebra} {\bf 15} (2006) 329--336.

\bibitem{SCH} I. Schur, \"{U}ber eine Klasse von Mittelbildungen mit Anwendungen auf die Determinantentheorie, {\it Sitzungsber. Berl. Math. Ges.} {\bf 22} (1923) 9--20.

\bibitem{SMITH} J. H. Smith, {\it Some properties of the spectrum of a graph, Combinatorial Structures and their Applications}, Gordon and Breach, New York, 1970. 

\bibitem{SAGE} W.A. Stein, Sage Mathematics Software (Version 9.5), The Sage Development Team, http://www.sagemath.org, 2015.

\bibitem{SMD1} S. Sun, Y. Min, K. C. Das, Extremal graphs for the sum of two largest eigenvalues,
{\it AIMS Mathematics} {\bf 11} (5) (2026) 15028--15036.

\bibitem{SMD2} S. Sun, Y. Min, K. C. Das, Sum of the $k$ largest eigenvalues of symmetric matrices: theory and applications, arXiv:2605.26707 (30 Pages)

\bibitem{WS} W. Wang, W. So, Graph energy change due to any single edge deletion, {\it Electron. J. Linear Algebra} {\bf 16} (2007) 291--299.

\bibitem{WWBBW} J. Wang, J. Wang, M. Brunetti, F. Belardo, L. Wang, Developments on the Hoffman program of graphs, {\it Adv. Appl. Math.} {\bf 169} (2025) 102915.

 \end{thebibliography}
\end{document}